\documentclass[a4paper,12pt]{amsart}
\usepackage[dvipdfm]{hyperref}
\usepackage{mathrsfs,rawfonts}
\usepackage{setspace, fullpage}
\usepackage{latexsym,amssymb,amsmath,mathrsfs,amsbsy, amsthm}

 \newtheorem{thm}{Theorem}[section]
 \newtheorem{lem}[thm]{Lemma}

 \newtheorem{cor}[thm]{Corollary}
 \newtheorem*{llem}{Reduction Lemma}
 \newtheorem*{defn}{Definition}
 \newtheorem{rem}{Remark}[section]
 \numberwithin{equation}{section}
 \newcommand{\sch}{Schr\"odinger }

 \newcommand{\rk}{\mathbb{R}^K}

 \newcommand{\dif}[1]{\frac{d}{d{#1}}}

 \newcommand{\Real}{\mathbb{R}}
 
 \newcommand{\Complex}{\mathbb{C}}
 
 \newcommand{\abs}[1]{\vert#1\vert}
 \newcommand{\norm}[1]{\Vert#1\Vert}
 \newcommand{\Norm}[1]{\left\Vert#1\right\Vert}
 \newcommand{\ti}{\tilde}
 \newcommand{\ut}{\frac {\partial u} {\partial t}}
 \newcommand{\p}{\partial}
 
 \newcommand{\Ld}{\Lambda}
 
 \newcommand{\ep}{\epsilon}
 \def\lg{\langle} \def\rg{\rangle}

\begin{document}

\title{\bf Schr\"odinger Soliton from Lorentzian Manifolds}

\author{Chong Song}
\email{songchong@amss.ac.cn}
\address{LMAM, School of Mathematical Sciences,Peking University,
Beijing 100871, P.R. China.}

\author{Youde Wang}
\email{wyd@math.ac.cn}
\address{Academy of Mathematics and System Sciences,Chinese Academy
of Science, Beijing 100190, P.R. China.}
\thanks{Partially supported
by 973 project of China, Grant No. 2006CB805902.}


\maketitle

\begin{abstract}
In this paper, we introduce a new notion named as \sch soliton. So-called \sch
solitons are defined as a class of special solutions to the \sch
flow equation from a Riemannian manifold or a Lorentzian manifold $M$
into a K\"ahler manifold $N$. If the target manifold
$N$ admits a Killing potential, then the \sch soliton is just a
harmonic map with potential from $M$ into $N$. Especially, if the
domain manifold is a Lorentzian manifold, the \sch soliton is a
wave map with potential into $N$. Then we apply the geometric energy
method to this wave map system, and obtain the local well-posedness
of the corresponding Cauchy problem as well as global existence in $1+1$
dimension. As an application, we obtain the existence of \sch soliton of
the hyperbolic Ishimori system.
\end{abstract}

\section{Introduction}

In this paper we intend to study a class of special solutions of the
\sch flows from a Riemannian manifold or a Lorentzian manifold into
a K\"ahler manifold. First, let us recall some preliminaries on \sch
flows. Let $(M,g)$ be a Riemannian manifold or a Lorenzian manifold
and $(N,h,J)$ be a K\"ahler manifold, where $J$ denotes the complex
structure and $h$ is the K\"ahler metric. The \sch flow is a map
$w:\Real\times M \to N$ which satisfies the equation
\begin{equation}\label{sch}
    \left\{
    \begin{aligned}
    &\frac {\p w}{\p t} = J(w)\tau(w),\\
    &w(0) = w_0.
    \end{aligned}
    \right.
\end{equation}
where $\tau(w) = trace_g\nabla^2w$ is the tension field of $w$, and
$w_0$ is an initial map from $M$ to $N$.

The \sch flow from a Riemannian manifold stems from fluid mechanics
and physics. It is a problem with strong physical backgrounds and a
long history. A century ago Italian mathematician Da Rios studied
the motion behavior of vortex filament and discovered the well-known
Da Rios equation which can be formulated as
 $$\gamma_t=\gamma_s\times\gamma_{ss},$$
where $\gamma(s,t): S^1\times \Real\rightarrow \Real^3$ is a closed
space curve for a fixed time $t$. By differentiating the above
equation with respect to $s$ we obtain the so called ferromagnetic
spin chain system which is just the \sch flow into $S^2$.  For the
existence theory of \sch flow from a Riemannian manifold, we refer
to \cite{BIKT, Uh, Q, D, DW, PWW, RS, SSB} and references therein.
Yet, for the \sch flow from Lorentzian manifolds, little is known.
In 1984, Ishimori~\cite{I} proposed a model as a 2 dimensional
analogue of the classic continuous isotropic Heisenberg spin chain,
which also describes the evolution of a system of static spin
vortices in the plane. The hyperbolic-elliptic Ishimori problem is a
spin field model with the form:
\begin{equation}\label{e:Ish}
    \left\{
    \begin{aligned}
    \p_t s &= s\times \square s + b(\p_x s \cdot \p_y \phi + \p_y s \cdot \p_x
    \phi),\\
    \Delta \phi &= 2s\cdot(\p_x s\times \p_y s),
    \end{aligned}
    \right.
\end{equation}
where $s:\Real^2\times \Real \to S^2\hookrightarrow\Real^3$,
$\square = \p_x^2 - \p_y^2$, $b\in\Real$ and $\lim_{|x|,|y|\to \infty} s(x,y,t) = (0,0,-1)$.
The Cauchy problem associated to Ishimori system~(\ref{e:Ish}) has
been studied extensively in the past decades, see for
example~\cite{K4, S} and references therein. When $b=0$, this system
gives a simple example of \sch flow from a Lorentzian manifold.

Kenig, Ponce and Vega \cite{K3} have ever studied the following \sch
equation which is analogous to the \sch flow from Lorentzian
manifold:
\begin{equation}
    \left\{
    \begin{aligned}
    &\ut = i\mathscr{L}u + P(u, \nabla u, \bar{u}, \nabla \bar u),\\
    &u(0) = u_0.
    \end{aligned}
    \right.
\end{equation}
where $u = u(t,x)$ is a complex valued function from
$\Real\times\Real^n$, $\mathscr{L}$ is a non-degenerate second-order
operator
$$\mathscr{L} = \sum_{i\le k}\p_{x_i}^2 - \sum_{j>k}\p_{x_j}^2 $$
for some $k\in \{1,\cdots, n\}$, and $P:\Complex^{2n+2} \to \Complex
$ is a polynomial satisfying certain constraints. They proved the
local well-posedness of the above initial value problem in
appropriate Sobolev spaces.

Since it is difficult to establish a general existence theory for
\sch flow from Lorentzian manifolds, we return to looking for some
special solutions. We recall that in \cite{DY} the authors proposed
to study the periodic solutions of the \sch flow in the case where
the target manifold $N$ is a K\"ahler-Einstein manifold with
positive scalar curvature. If the target manifold is just the
standard sphere $S^2$, they employed the well-known symmetric
variational principle to show the existence of some special periodic
solutions to the flow from a closed base surface with convolution
symmetry. In particular, they needed to reduce the \sch flow to a
elliptic equation and established the following lemma on reduction.

\begin{llem}\label{lem:red}
Assume there exists a non-trivial holomorphic Killing vector field
$V$ on $N$, and let $S_t$ be the one-parameter group of holomorphic
isometries generated by $V$ with $S_0 = I$, the identity map. Then
$w(t) = S_t\circ u$ with $u:M\to N$ is a solution to (\ref{sch}) if
and only if $u$ is a solution to the equation
\begin{equation}\label{equ:1}
\tau(u) = -J(u)V(u).
\end{equation}
\end{llem}

\begin{proof}
Directly computing by the definition of tension field, we get
\begin{equation*}
  \tau(w) = \tau(S_t\circ u) = dS_t\circ \tau(u) + \tau(S_t)(du, du).
\end{equation*}
Since $S_t$ is an isomorphism, we have $\tau(S_t) = 0$ and hence
\begin{equation*}
  \tau(w) = dS_t\circ\tau(u).
\end{equation*}

On the other hand,
\begin{equation*}
  w_t = \frac \p {\p t}(S_t\circ u) = V(S_t\circ u) = dS_t\circ V(u).
\end{equation*}
The last equality holds because the single parameter group $S_t$
satisfies $S_t\circ S_s = S_{t+s}$. Differentiating this at $s=0$,
we get $dS_t\circ V = V(S_t)$.

Next, because $V$ is holomorphic, i.e. $[J,\nabla V]=0$, we have
\[J\circ dS_t=dS_t\circ J.\]
Combining above equalities together, we arrive at
\begin{equation}\label{e0}
  w_t = dS_t\circ V(u) = J(w)\tau(w) = J(S_t\circ u)dS_t\circ \tau(u) = dS_t\circ J(u)\tau(u).
\end{equation}
$dS_t$ is an isomorphism on the tangent space, so ~(\ref{e0}) is
equivalent to ~(\ref{equ:1}).
\end{proof}

It is easy to see that the special solution to \sch flow given by
the above lemma is some kind of solitary wave solution. In fact, for
a linear \sch equation defined on a flat torus $\mathbb{T}^m$
 $$iw_t=\Delta w,$$
a solitary wave solution is of the form $w=ue^{ikt}$ where $k$ is a
positive constant, and $v$ is a real function which satisfies the
equation $\Delta u + ku =0$. Here, $e^{ikt}$ can be viewed as a
holomorphic isometric group with one parameter. Therefore, we define
the \sch soliton as follows

\begin{defn}
Suppose $u$ is a solution to (\ref{equ:1}) derived in the Reduction
Lemma, then $w(t) = S_t\circ u$ is called a \sch soliton solution of
(\ref{sch}).
\end{defn}

A solution to equation (\ref{equ:1}) is a map with prescribed
tension field. In general it is hard to solve the equation because
the elliptic system is not of a variational structure. There are
only a few results under some strong assumptions, see \cite{CJ} for
example.

However, if there exists a smooth function $\Ld\in C^\infty(N)$ on
$N$, such that $JV = \nabla \Ld$ is the gradient vector field of
$\Ld$, then the equation becomes
\begin{equation}\label{e7}
  \tau(u) = -\nabla \Ld(u),
\end{equation}
and it's easy to see that this equation is the Euler-Lagrange
equation of the following functional:
\begin{equation}\label{e7'}
  F(u) = E(u) - \int_M \Ld(u) dV_g.
\end{equation}
Here
 $$E(u) = \frac{1}{2}\int_M \abs{\nabla u}^2 dV_g$$
is the energy functional of maps $u \in W^{1,2}(M, N)$, where
${|\nabla u|}^2 = trace_g(u^*h)$. \\

In the case $M$ is a Riemannian manifold, the solutions to
equation (\ref{e7}) are harmonic maps with potential $\Ld$ from $M$
into $N$. Once we have the above variational structure, many
powerful tools which are adopted to study harmonic maps work for the
present problem and many results on harmonic maps can be extended.
For formal results on harmonic maps with potential, we refer to
~\cite{C1,C2,FR,FRR}.

In this paper, however, we focus on the situation where the base
manifold is Lorentzian. It is well-known that the hyperbolic
harmonic maps from a Lorentzian manifold are usually called wave
maps and the well-posedness of wave maps has been intensively
studied by many mathematicians; see for example \cite{SS, ST, T1}
and many references therein. We will see below that the \sch soliton
from a Lorentzian manifold (or Lorentzian \sch soliton for short)
satisfies a perturbed wave map equation. It's worthy pointing out that
this kind of wave map with potential emerges
naturally as a simplified equation of the dynamics of weak
ferromagnets magnetization when $N =S^2$~\cite{HS}.

Indeed, let $(M_1, g_1)$ be a compact Riemannian manifold with the
Riemannian metric $g_1 = g_{\alpha\beta}dx^\alpha dx^\beta$ and $M =
\Real \times M_1^m$ be a Lorentzian manifold equipped with a
Lorentzian metric $g = dt^2 - g_1$. Denote the covariant derivative
for functions on $M_1$ and $M$ by $\nabla$ and $\ti \nabla$
respectively. We will always embed the compact target manifold $N$
into a Euclidean space $\Real^K$.  Then the equation (\ref{equ:1})
becomes
\begin{equation}\label{equ:2}
\square u = A(u)(\ti \nabla u,\ti \nabla u) - J(u)V(u),
\end{equation}
where $\square = \p^2_t - \Delta$ is the wave operator, $\ti \nabla
u = u_t + \nabla u$ and $A(u)(\cdot, \cdot)$ is the second
fundamental form of $N\subset \Real^K$. Using the Christoffel
symbols $\Gamma^k_{ij}$ of $N$, one can write explicitly in local
coordinates that
\[ (A(u)(\ti \nabla u,\ti \nabla u))^k = \Gamma^k_{ij}u^i_t u^j_t -
  g^{\alpha\beta}\Gamma^k_{ij}\nabla_\alpha u^i \nabla_\beta u^j. \]
Equation~(\ref{equ:2}) is a nonlinear wave system. In particular, if
there exists a Killing potential (See Section \ref{killing} for the
definition) $\Ld \in C^\infty(N)$ such that $JV = \nabla \Ld$, the
equation becomes
\begin{equation}\label{equ:2'}
  \square u = A(u)(\ti \nabla u,\ti \nabla u) - \nabla \Ld(u).
\end{equation}
We will call a solution to equation~(\ref{equ:2'}) a wave map with
potential. We will consider initial data
\begin{equation}\label{equ:ini}
(u(0),u_t(0)) = (u_0,u_1); ~~u_1(x) \in T_{u_0(x)}N, \text{~a.e.~}
x\in M_1
\end{equation}
and study the corresponding Cauchy problem. Our main result is the
following theorem:

\begin{thm}\label{thm:sch}
Suppose $(M_1, g_1)$ is an $m$-dimensional compact Riemannian
manifold with $m>1$ and $M = \Real \times M_1$ is equipped with a
Lorentzian metric $g = dt^2 - g_1$, let $N$ be a compact K\"ahler
manifold with a Killing potential $\Lambda$ such that $\nabla \Ld =
JV$. Let $k\ge m_0 = [\frac m 2]+1$, where $[\frac m 2]$ denotes the
integer part of $\frac m 2$. Than for initial maps $u_0\in
W^{k,2}(M_1,N)$ and $u_1\in W^{k-1,2}(M_1,T_{u_0}N)$, the Cauchy
problem (\ref{equ:2'}), (\ref{equ:ini}) has a unique local solution
$u$ satisfying $u\in L^\infty([0,T),W^{k,2}(M_1,N))$ and $u_t\in
L^\infty([0,T),$  $W^{k-1,2}(M_1,T_uN))$. Moreover, if the initial
data is smooth, so is the solution.
\end{thm}

\begin{rem}
Although for the sake of consistency with the \sch soliton, we only
discuss wave maps with Killing potentials in this paper, by exactly
the same procedure one can verify that Theorem~\ref{thm:sch} holds
for wave maps with any potential $\Ld$, i.e. for any smooth function
$\Ld:N\to \Real.$
\end{rem}

In the classical wave map theory, it has been shown that the Cauchy
problem of wave map is locally well-posed on Minkowski space
$\Real\times\Real^n$ with initial data $(u_0, u_1)\in
W^{k,2}(\Real^n,N)\times W^{k-1,2}(\Real^n,T_{u_0}N)$, where $k =
\frac{n+1}2$ for $n\ge 3$ and $k>\frac 32$ for $n=2$.(See
Theorem~7.2 in~\cite{SS}.) On the other hand, the
$C^\infty$-regularity of wave equations is well-known by the theory
of paradifferential operators. Thus Theorem~\ref{thm:sch} is a
generalization of the well-known results for wave maps to the
current perturbed wave map system on Lorentzian manifolds. Note that
$m_0$ is the critical exponent on the manifold $M$, since there are
no fractional Sobolev spaces on manifolds.

This generalization won't take much effort since the perturbing term is of
lower order. However, in this paper, we employ a new method, namely, the geometric
energy method which first appeared in Ding and Wang's
work~\cite{DW} to tackle this problem. It's worthy to point out that the geometric energy
method is a powerful tool in dealing with various kinds of geometric
evolution equations. It's also the first time shown in this paper that
the wave map (with potential) can be handled by this method. It
provides a simplified and uniform method which avoids the
complicated analysis of fixing moving frames, choosing Columb gauge,
etc. (See ~\cite{SS1} for example.)

Another advantage of this method is that we can directly obtain the
$C^\infty$-regularity of the solution to the Cauchy problem with
smooth initial data. The fact that the Cauchy
problem is locally well-posed with initial data in $W^{k,2}$ for all
$k\ge m_0$ dose not directly imply the local well-posedness in
$C^\infty$. Because the space of smooth maps
$C^\infty(M_1,N)= \cap_{k=m_0}^\infty W^{k,2}(M_1,N)$ itself is not
a Banach space, and the standard techniques such as fixed point
theory do not apply here. Our method provides an uniform lower bound of
the maximal time $T_k$ for all $k\ge m_0+1$, see Lemma~\ref{lem:6} below.
With this bound, we are able to assert the existence of a local
solution $u\in C^\infty([0,T)\times M_1, N)$ to the Cauchy problem
with smooth initial data. Moreover, the maximal time $T$ only
depends on the geometry of $N$, $\norm{u_0}_{W^{m_0+1,2}}$ and
$\norm{u_1}_{W^{m_0,2}}$, see Theorem~\ref{thm:smooth}.

In addition, we prove the global existence of solution to the
Cauchy problem~(\ref{equ:2'}), (\ref{equ:ini}) on $1+1$ dimensional
Lorentzian manifolds. This is an analogous result to the wave map
theory, see \cite{G} and \cite{SS}.

\begin{thm}\label{thm:global}
Let $M_1\equiv S^1$ and $N$ be a compact Riemannian manifold.
Suppose $\Lambda$ is a smooth function on $N$ and $u_0\in
W^{2,2}(S^1,N)$, $u_1\in W^{1,2}(S^1,T_{u_0}N)$, then the Cauchy
problem (\ref{equ:2'}), (\ref{equ:ini}) has a unique global solution
$u\in L^\infty(\Real^+,W^{2,2}(S^1, N))$.
\end{thm}

Therefore, in this case by Reduction Lemma, we get a \sch soliton
solution $w(t) = S_t\circ u$ to the \sch flow~(\ref{sch}) with
initial data $w(0) = u$, which is a special global solution.

\begin{cor}\label{thm:soliton}
Let $M=\mathbb{R}\times S^1$ be a Lorentzian manifold and $N$ a
compact K\"ahler manifold with a Killing potential $\Lambda$. Then
given $u_0\in W^{2,2}(S^1,N)$ and $u_1\in W^{1,2}(S^1,T_{u_0}N)$,
there exists a \sch soliton solution of~(\ref{sch}).
\end{cor}

\begin{rem}
Theorem~\ref{thm:sch} also holds for Lorentzian space
$M=\Real^{n+1}$. Similarly, Theorem~\ref{thm:global} and
Corollary~\ref{thm:soliton} hold true if we replace $S^1$ by
$\Real^1$. One may check this by following the argument in
\cite{DW}. The crucial fact is that the interpolation inequality in
Theorem~\ref{thm:int} is scaling invariant, and hence the main
estimate in Lemma~\ref{lem:6} does not depend on the diameter of the
domain.
\end{rem}

Particularly, since there is an natural Killing potential on $S^2$(see
Section~2), Corollary~\ref{thm:soliton} provides a global solution
to the Ishimori system~(\ref{e:Ish}) when $b=0$. Indeed, if $S^2 \in
\Real^3$ is the standard sphere and $\Ld(u) = u_3$ denotes the
projection of $u = (u_1,u_2,u_3)\in S^2$ to the $z$-axis in
$\Real^3$, then $\Ld$, which is the first eigenfunction of $S^2$, is
a Killing potential. The gradient field of $\Ld$ is given by
\[ \nabla \Ld(u) = P(u)e_3, \]
where $e_3 =(0,0,1)$ is the unit vector and $P(u)$ denotes the
orthogonal projection of $\Real^3$ to $T_uS^2$. Precisely,
\[ P(u)e_3 = e_3 - (u,e_3)u. \]
The complex structure of $S^2$ is $J(u) = u\times$. Thus the
corresponding Killing field $V$ is
\[ V(u) = -J(u)\nabla \Ld(u) = -u\times P(u)e_3 = -u\times e_3, \]
and the isometry family $S_t$ is just a rotation around the
$z$-axis. Then if $s(t,x,y) = S_t\circ u(x,y)$, by Reduction Lemma,
the equation of $s$ in Ishimori system~(\ref{e:Ish}) can be reduced
to
\begin{equation}\label{e00}
\square u = (|\p_x u|^2-|\p_y u|^2)u - P(u)e_3.
\end{equation}
Thus given initial data $(u_0,u_1)$, we get a solution $u$ on
Lorentzian space $\Real^{1+1}$ to equation~(\ref{e00}) by
Theorem~\ref{thm:global}. Hence $s(x,y,t) = S_t\circ u(x,y)$ gives a
global periodic solution to the initial value problem of Ishimori
system~(\ref{e:Ish}) with $s(0) = u$. We call such a solution $s$ 
a \sch soliton solution of Ishimori system~(\ref{e:Ish}). Therefore,
we obtain the following corollary.

\begin{cor}
If $b=0$, then given $u_0\in W^{2,2}(\Real^1,N)$ and $u_1\in
W^{1,2}(\Real^1,T_{u_0}N)$, there exists a \sch soliton solution of
Ishimori system~(\ref{e:Ish}).
\end{cor}

The rest of the this paper is organized as follows: in Section 2 we
briefly introduce the Killing potential; in Section 3 we prove
Theorem~\ref{thm:sch}; finally we prove Theorem~\ref{thm:global} and
hence Corollary~\ref{thm:soliton} and 1.4 in Section 4.

\section{Killing potential and some remarks}\label{killing}

We know that the \sch soliton equations are not of variational
structure generally. So, it is very difficult to solve
(\ref{equ:1}), since the classical variational methods can not be
used to approach this problem. In fact, it may do not admit any
solution at all. Then a natural question is: \emph{when dose the
equation ~(\ref{equ:1}) have a variational structure?} One has found
the question relates closely to whether a K\"ahler manifold admits a
Killing potential function or not. Therefore, let's recall the
notion of Killing potential as follows.

\begin{defn}
If $\Ld$ is a smooth function on a K\"ahler manifold $(N,J)$, and the gradient field of $\Ld$ has the form:
\[ \nabla \Ld = JV, \]
where $V$ is a Killing field on $N$, then $\Ld$ is called a Killing potential.
\end{defn}

Obviously, if there exists a Killing potential on $(N, J)$, then
(\ref{equ:1}) is of the desired variational structure. Now, a
question confronting us is what kind of manifolds do admit Killing
potentials? Fortunately, one has made great progress on the
existence of Killing potentials on a K\"ahler manifold in
differential geometric field. Recently, Derdzinski and Maschler
studied the so-called special K\"ahler-Ricci potentials which is a
special kind of Killing potential, and gave a local classification
for the K\"ahler manifolds admitting such potentials. It's also
related to the conformally-Eintein K\"ahler metrics. One can refer
to \cite{DM1, DM2, Je} for more details.

For completeness, here we give several basic lemmas about Killing
potential.

\begin{lem}\label{f2} (\cite{DM2})
  Suppose $\Ld$ is a smooth function on a K\"ahler manifold, then the following conditions are equivalent:
  i) $\Ld$ is a Killing potential; ii) $\nabla \Ld$ is a holomorphic vector field; iii) $\nabla^2\Ld$ is Hermitian.
\end{lem}
\begin{proof}
  Let $V = -J\nabla \Ld$, then $\Ld$ is a Killing potential is equivalent to say $V$ is a Killing potential, which means $\nabla V$ is skew-symmetric,
  i.e.
  \begin{equation}\label{f1}
    (\nabla V)^* + \nabla V = 0.
  \end{equation}
  Since $\nabla V = -J\nabla^2\Ld$, $(\nabla^2\Ld)^* = \nabla^2\Ld$ and $J^* = -J$, (\ref{f1}) is equivalent to
  \[ \nabla^2\Ld\circ J - J\circ \nabla^2\Ld = [\nabla^2\Ld, J] = 0, \]
  which means $\nabla \Ld$ is holomorphic. Thus i) and ii) are equivalent.
  On the other hand, if $\nabla^2\Ld$ is Hermitian, i.e. $\nabla^2\Ld(X,JY) = -\nabla^2\Ld(JX,Y)$ for any vector fields $X,Y$.
  Then
  \[ \nabla^2\Ld(X,JY) = \lg X, \nabla_Y J\nabla \Ld \rg = -\nabla^2\Ld(JX,Y) = -\lg \nabla_X J\nabla \Ld, \nabla_Y \rg. \]
  This is equivalent to the skew-symmetry of $V = -J\nabla \Ld$, which is equivalent to i).
\end{proof}

\begin{lem}(\cite{DM2})
Suppose $(N,h,J)$ is a K\"ahler manifold. If $H_1(N,\Real) = 0$,
then for every holomorphic Killing field $V$ there exists a Killing
potential $\Ld$, such that $\nabla \Ld = JV$.
\end{lem}
\begin{proof}
  Since $V$ is Killing and holomorphic, $\nabla V$ is skew symmetric and commutes with $J$. Thus if we let $W = JV$, then $\nabla W$
  is symmetric. This implies the corresponding 1-form $\xi = \iota_W h$ is closed, since
  $$ (d\xi)(X,Y) = h(\nabla_X W, Y) - h(X, \nabla_Y W) $$
  for any vector fields $X,Y$. So there exist a function $\Ld$ such that $d\Ld = \xi$ and hence $\nabla \Ld = W = JV$.
\end{proof}

In fact, the existence of Killing potential is a complicated
problem and somehow related to the topology of the underlying
manifold. The following lemma gives a sufficient condition for the existence of Killing potential:

\begin{lem}
  Let $\Ld$ be a $C^\infty$ funcion on a K\"ahler manifold $(M,g)$ such that
  \begin{equation}\label{e:killing}
   \nabla^2\Ld + \chi Ric = \sigma g,
  \end{equation}
  where $Ric$ is the Ricci tensor, and $\chi, \sigma$ are some $C^\infty$ functions.
  Then $\Ld$ is a Killing potential.
\end{lem}
\begin{proof}
  It is a direct corollary from iii) of lemma~\ref{f2} and the fact that $Ric$ and $g$ are Hermitian.
\end{proof}

From this lemma, one can see that there are plenty of manifolds
admitting Killing potentials, including special cases of independent
interest. For example, compact K\"ahler manifolds with function
$\Ld$ satisfying (\ref{e:killing}) for constants $\chi, \sigma$ such
that $\chi\sigma>0$ are known as K\"ahler-Ricci solitons (\cite{P},
\cite{TZ}). Also, Riemannian manifolds admitting functions $\Ld$
satisfying (\ref{e:killing}) with $\chi=0$ have been studied
extensively, and their local structure is completely understood in
\cite{K2}.

We know that it is always an important issue that how many closed
geodesics exist on a compact Riemannian manifold. An one-dimensional
\sch solitons from $S^1$ into a compact K\"{a}hler manifold with a
Killing potential $\Lambda$ is a geodesic with potential. Since
$\Lambda$ is closely relevant to the geometry and topology of the
target manifold, it is of significance that we study the existence
of such geodesics. Naturally, we may ask the following
\begin{quote}
{Question 1:} \emph{At least how many closed geodesics with
potential $\Lambda$ exist on a closed K\"ahler manifold with Killing
potential?}
\end{quote}

On the other hand, we should mention another important special case.
When $N$ is a compact K\"ahler-Einstein manifold with positive
scalar curvature, it is known that for every Killing field $V$,
$JV=\nabla \Ld_1$ is the gradient vector field of the first
eigenfunction $\Ld_1$ of the Laplace-Beltrami operator $\Delta_N$ on
$N$ (\cite{K1}). By virtue of this fact and Sacks-Uhlenbeck's
perturbed technique, Ding and Yin \cite{DY} proved there exists an
infinite number of inequivalent periodic solutions to the \sch flow
(periodic \sch solitons) from $S^2$ into $S^2$ (see also \cite{HW}).
In this case the potential function in the above Question 1 is just
the first eigenfunction on $N$. In fact, more generally we may
consider the following
\begin{quote}
{Question 2:} \emph{Let $N$ be a closed Riemannian manifold and
$\Lambda_1(x)$ be the first eigenfunction of the Laplace-Beltrami
operator $\Delta_N$. At least how many closed geodesics with
potential $\Lambda_1(x)$ exist on $N$?}
\end{quote}

\section{Local well-posedness}


In this section, we will use the geometric energy method in
\cite{DW} to prove the local well-posedness of Lorentzian \sch
solitons into a compact K\"ahler manifolds with a Killing potential
and wave maps with potential. We need to recall an important theorem
proved in \cite{DW}. This is a generalized Gagliardo-Nirenberg
inequality.

Let $\pi:E \to M_1$ be a Riemannian vector bundle over an
$m$-dimensional Riemannian manifold $M_1$ and let $D$ denote the
covariant derivative on $E$ induced by the Riemannian metric. Then
we can define a Sobolev norm via the bundle metric for every section
$s\in \Gamma(E)$ by
\[ \norm{s}_{H^{k,q}} = \sum_{l=0}^k \norm{D^l s}_{L^q}. \]

\begin{thm}\label{thm:int} (\cite{DW})
Suppose $s \in C^\infty(E)$ is a section where $E$ is a vector
bundle on $M_1$. Then we have
\begin{equation}
\Norm{D^j s}_{L^p} \leq C\Norm{s}^a_{H^{k,q}}\Norm{s}^{1-a}_{L^r},
\end{equation}
where $1\leq p,q,r\leq \infty$, and $j/k \leq a \leq 1 (j/k \leq a <
1$ if $q=m/(k-j) \neq 1)$ are numbers such that
\[ \frac 1 p = \frac j m + \frac 1 r + a(\frac 1 q - \frac 1 r - \frac k m). \]
The constant $C$ only depends on $M_1$ and the numbers $j,k,q,r,a$.
\end{thm}

For Lorenzian manifold $M = \Real\times M_1$ with metric $g = dt^2 -
g_1$ and the compact manifold $N$ which is embedded into $\Real^K$,
let $D$ denote the covariant derivative on the pull-back tangent
bundle $u^*(TN)$ over $M_1$ of $u \in C^\infty (M_1,N)$ and $\tilde
D = D_t + D$ denote the covariant derivative on the bundle over $M$.
Recall we also use $\nabla$ and $\tilde \nabla$ to denote the
covariant derivative of functions on $M_1$ and $M$ respectively. For
convenience we denote $Du = \nabla u$ and $\tilde Du = \tilde \nabla
u$. Obviously, $D^2 u = (\nabla^2 u)^\top$ is the tangent part of
$\nabla^2 u$.

Then by the theorem, for $Du \in \Gamma(u^*(TN))$, we have
\begin{equation}\label{equ:int}
\Norm{D^{j+1} u}_{L^p} \leq C\Norm{D u}^a_{H^{k,q}}\Norm{D u}^{1-a}_{L^r}.
\end{equation}
Ding and Wang also showed that the $H^{k,p}$ norm of section $Du$ is equivalent to the normal Sobolev $W^{k+1, p}$ norm of the
map $u$. Precisely, we have

\begin{lem}\label{lem:equ} (\cite{DW})
Assume that $k > m/2$. Then there exists a constant $C = C(N,k)$
such that for all $u \in C^\infty(M_1,N)$,
\[ \Norm{\nabla u}_{W^{k-1,2}} \leq C\sum_{i=1}^k \Norm{Du}^i_{H^{k-1,2}} \]
and
\[ \Norm{Du}_{H^{k-1,2}} \leq C\sum_{i=1}^k \Norm{\nabla u}^i_{W^{k-1,2}} \]
\end{lem}


Now we return to the equation (\ref{equ:1}), using the covariant derivative $D$, we can rewrite the equation:
\begin{equation}\label{e1}
\tau(u) = \text{trace}_g (\tilde{D}^2u) = D^2_t u - \sum_{\alpha = 1}^m D_\alpha D_\alpha u = -J(u)V(u).
\end{equation}

To prove the existence of the above equation, usually one needs to
choose a suitable approximate equation for which the existence is
easy to prove, and some uniform a priori estimates of solutions with
respect to the parameter $\epsilon$ needs to be established. Here we
follow \cite{Z} due to Y. Zhou and use the viscous approximation
\begin{equation}\label{equ:3}
D^2_t u - D_\alpha D_\alpha u - \epsilon D_\alpha D_\alpha u_t =
-J(u)V(u),
\end{equation}
where $\epsilon>0$ is a small parameter. Or equivalently,
\begin{equation}\label{equ:4}
u_{tt} - \epsilon\Delta u_t -  \Delta u + J(u)V(u) = A(u)(\tilde
\nabla u,\tilde \nabla u) - \epsilon T(u)(\Delta u_t) \bot T_uN,
\end{equation}
where $T(u)$ denotes the orthogonal projection to the normal bundle
at $u$, i.e.
$$ T(u)(\Delta u_t) = \Delta u_t - (\Delta u_t)^\top. $$
We already know that
\begin{align*}
(\Delta u_t)^\top &= \text{trace}_{g_1}D^2 u_t\\
&=\text{trace}_{g_1}D(\nabla u_t - A(u)(u_t))\\
&= \Delta u_t - A(u)(\nabla u_t, \nabla u) - \text{div} (A(u)(u_t,
\nabla u)).
\end{align*}
Thus we have
\begin{equation}\label{a}
T(u)(\Delta u_t) = A(u)(\nabla u_t, \nabla u) + \text{div}
(A(u)(u_t, \nabla u)).
\end{equation}

This equation (\ref{equ:4}) may be viewed as a parabolic system for
$u_t$. Indeed, the local existence and uniqueness of smooth
solutions to (\ref{equ:4}) for initial data $(u_0,u_1) \in C^\infty
(M_1,TN)$ such that
\begin{equation}\label{equ:ini'}
(u, u_t)(\cdot,0) = (u_0,u_1)
\end{equation}
can be derived by a fixed point argument using the heat kernel of
$M_1$ (see the appendix). Actually, M\"uller and Struwe \cite{MS}
used this approximation method to prove the global existence of weak
solutions to the wave map equation
 in $1+2$ dimensions with finite energy data.

We can define the energy density for a map $u:M\to N$ and $\forall
t\in \Real$ by
\[ e(t) := \frac 1 2 \abs{\ti \nabla u(t)}^2, \]
where
\[ |\ti \nabla u(t)|^2 = \abs{u_t(t)}^2 + \abs{\nabla u(t)}^2. \]
Notice that the norm here is different from the norm induced by the
Lorentzian metric $g = dt^2 - g_1$. This is a convention in wave map theory which we
will adopt through out this paper.

Now we define the energy functional for all maps $u\in W^{1,2}(M,N)$
and $\forall t\in \Real$ by
\[ E(t) := \int_{\{t\}\times M_1}e(t) dV_{g_1},\]
For this energy functional, we have the following energy inequality:

\begin{lem}\label{lem:6}
  For any $\ep\in(0,1]$, suppose $u \in C^\infty(M_1\times[0,T_\ep),N)$
  is a local solution to Cauchy problem~(\ref{equ:3}),~(\ref{equ:ini'}).
  Then we have
  \begin{equation*}
    E(t) \le E(0) - \int^t_0\int_{M_1}\lg u_t,J(u)V(u)\rg.
  \end{equation*}
  Particularly, if $JV = \nabla \Ld$ is the gradient field of a Killing potential $\Ld$, we have
  \begin{equation}\label{eq:12}
    E(t) \le E(0) - \int_{M_1}\Ld(u(t)) + \int_{M_1}\Ld(u(0)).
  \end{equation}
\end{lem}
\begin{proof}
  Using the equation~(\ref{equ:4}), we have
  \begin{align*}
    \frac {dE(t)}{dt} &= \int_ {M_1}\lg u_{tt}, u_t\rg + \lg \nabla u, \nabla u_t \rg \\
    &= \int_{M_1}\lg u_{tt}, u_t\rg - \lg \Delta u, u_t \rg\\
    &= \int_{M_1}\lg \ep\Delta u_t - J(u)V(u) + A(u)(\tilde \nabla u,\tilde \nabla u)
    - \epsilon T(u)(\Delta u_t), u_t \rg\\
    &= -\ep \int_{M_1}|\nabla u_t|^2 - \int_{M_1}\lg J(u)V(u),u_t \rg \\
    &\le - \int_{M_1}\lg J(u)V(u),u_t \rg.
  \end{align*}
Integrating this equality from $0$ to $t$, we get the lemma.
\end{proof}

Thus given a smooth initial data, we can get a local solution
$u_\epsilon \in C^\infty(T_\ep \times M_1,N)$ for every $\ep >0$
which satisfies the energy inequality. Next, in order to establish
the local existence of the equation~(\ref{e1}), we need to derive
some uniform a priori estimates for solutions $u_\epsilon$ with
respect to $\ep$. For this, we denote for a fixed time $t\in
[0,T_\ep)$
\[ \Norm{\tilde D u}^2_{L^2(M_1)} = \int_{M_1} \lg \ti Du, \ti Du \rg =
\int_{M_1} \lg D_tu, D_tu \rg + \lg Du, Du \rg. \] Note again this
norm is \emph{not} the one induce by the Lorentzian metric.

In the following we will assume $M_1$ is flat, i.e. the Riemannian
curvature of $M_1$ vanishes identically, to simplify the
computations. For the general case, the additional terms involving
the curvatures of $M_1$ actually do not provide additional
difficulties, since the derivatives of $u$ appearing in these terms
are of lower orders and the curvature of $M_1$ are bounded.

Let $\mathbf{a}$ be a multi-index with length $\abs{\mathbf{a}}=l$,
and $D_\mathbf{a}$ be the multi-derivative of space direction, we
compute
\begin{equation}
\frac 1 2 \dif{t} \Norm{D_\mathbf{a}\tilde D u}^2_{L^2(M_1)} = \int_{M_1}
\langle D_\mathbf{a}\tilde Du, D_tD_\mathbf{a}\tilde Du \rangle.
\end{equation}
Changing order of the covariant differentiation, we have
\begin{equation}\label{equ:5}
D_tD_\mathbf{a}\tilde Du
= D_\mathbf{a}D_t\tilde Du + \sum D_\mathbf{b}R(u)(D_\mathbf{c}u, D_\mathbf{d}D_tu)D_\mathbf{e}\tilde Du,
\end{equation}
where $R$ is the curvature tensor of $N$ and the summation is taken
for all multi-indexes $\mathbf{b},\mathbf{c},\mathbf{d},\mathbf{e}$
with possible zero lengths, except that $\abs{\mathbf c} > 0$ always
holds, such that
\[ (\mathbf{b},\mathbf{c},\mathbf{d},\mathbf{e}) = \sigma(\mathbf{a}) \]
is a permutation of $\mathbf{a}$. If we denote the curvature terms
like the second term on the right hand side of (\ref{equ:5}) by $Q$,
i.e.
\[Q(X,Y) = \sum D_\mathbf{b}R(u)(D_\mathbf{c}u, D_\mathbf{d}X)D_\mathbf{e}Y,\]
then we have
\begin{eqnarray}\label{e4}
\nonumber\frac 1 2 \dif{t} \Norm{D_\mathbf{a}\tilde Du}^2_{L^2(M_1)}
&=& \int_{M_1}
\langle D_\mathbf{a}D_t\tilde Du + Q_1, D_\mathbf{a}\tilde Du \rangle \\
&=& \int_{M_1}\langle D_\mathbf{a}D_t^2u, D_\mathbf{a}D_tu \rangle +
\langle D_\mathbf{a}D_tDu, D_\mathbf{a}Du \rangle + \langle Q_1,
D_\mathbf{a}\tilde Du \rangle,
\end{eqnarray}
where $Q_1 = Q(D_tu, \tilde Du)$.

For the second term in (\ref{e4}), we have
\begin{eqnarray}\label{e5}
\nonumber\int_{M_1} \langle D_\mathbf{a}D_tDu, D_\mathbf{a}Du
\rangle &=&
\int_{M_1}\langle DD_\mathbf{a}D_tu + Q_2, D_\mathbf{a}Du \rangle\\
\nonumber&=& -\int_{M_1} \langle D_\mathbf{a}D_tu, DD_\mathbf{a}Du \rangle + \langle Q_2, D_\mathbf{a}Du \rangle\\
\nonumber&=& -\int_{M_1} \langle D_\mathbf{a}D_tu, D_\mathbf{a}DDu + Q_3 \rangle + \langle Q_2, D_\mathbf{a}Du \rangle\\
&=& -\int_{M_1} \langle D_\mathbf{a}D_tu, D_\mathbf{a}DDu \rangle
-\langle D_\mathbf{a}D_tu, Q_3 \rangle + \langle Q_2, D_\mathbf{a}Du
\rangle
\end{eqnarray}
where $Q_2 = Q(Du,D_tu),Q_3 = Q(Du,Du)$.

To simplify the notations, we will put all the curvature terms $Q_i$
together and use $\tilde Q$ to denote the sum of those terms.

Combining~(\ref{e4}) and~(\ref{e5}) together and using the equation
(\ref{equ:3}), we get
\begin{align*}
\frac 1 2 \dif{t} \Norm{D_\mathbf{a}\tilde Du}^2_{L^2(M_1)} &\le
\int_{M_1}\langle D_\mathbf{a}D_t^2u - D_\mathbf{a}D Du,
D_\mathbf{a}D_tu \rangle
+|\tilde Q||D_\mathbf{a}\tilde Du| \\
&= \int_{M_1}\langle \epsilon D_\mathbf{a}DDu_t -
D_\mathbf{a}J(u)V(u), D_\mathbf{a}D_tu \rangle
+|\tilde Q||D_\mathbf{a}\tilde Du|\\
&= \int_{M_1}\langle \epsilon DDD_\mathbf{a}u_t + \epsilon Q_4 + \epsilon DQ_5 -
J(u)D_\mathbf{a}V(u), D_\mathbf{a}D_tu \rangle
+|\tilde Q||D_\mathbf{a}\tilde Du|\\
&= \int_{M_1}-\epsilon \langle DD_\mathbf{a}D_tu , DD_\mathbf{a}D_tu
\rangle - \langle J(u)D_\mathbf{a}V(u), D_\mathbf{a}D_tu \rangle
+|\tilde Q||D_\mathbf{a}\tilde Du|\\
&\le C\int_{M_1} \abs{D_\mathbf{a} u}\abs{D_\mathbf{a}D_tu}+
\abs{\tilde Q}\abs{D_\mathbf{a}\tilde Du},
\end{align*}
where $Q_4 = Q(Du,Du_t),Q_5 = Q(Du,D_tu)$. Obviously, we have
\begin{equation}\label{equ:curv}
\begin{split}
\abs{\ti Q} &\le |Q_1| + |Q_2| + |Q_3| +\epsilon|Q_4| + \epsilon|DQ_5|\\
&\le C\abs{Q(\tilde Du, \tilde Du)} + \epsilon\abs{Q(Du, Du_t)} + \epsilon\abs{DQ(Du, u_t)}\\
&\le C\sum\abs{D^{j_1}\tilde Du}\cdots\abs{D^{j_b}\tilde Du},
\end{split}
\end{equation}
where the summation is over all indexes $(j_1, \cdots, j_b)$ satisfying
\begin{equation}\label{equ:index}
j_1\ge j_2 \ge \cdots \ge j_b,\quad l\ge j_i \ge 0,\quad j_1 + \cdots + j_b + b \le l+3,\quad b \ge 3.
\end{equation}

Thus, we get
\begin{equation*}
\frac 1 2 \dif{t} \Norm{D_\mathbf{a}\tilde Du}^2_{L^2(M_1)} \le
 C\int_{M_1} \abs{D_\mathbf{a} u}\abs{D_\mathbf{a}D_tu} +
 C\sum\int_{M_1}\abs{D^l\tilde Du}\abs{D^{j_1}\tilde Du}\cdots\abs{D^{j_b}\tilde
 Du}.
\end{equation*}
Hence
\begin{equation}\label{eq:13}
\begin{aligned}
\frac 1 2 \dif{t} \Norm{D^l\tilde Du}^2_{L^2(M_1)} &\le
  C\int_{M_1} \abs{D^l u}\abs{D^lD_tu}+ C\sum\int_{M_1}
  \abs{D^l\tilde Du}\abs{D^{j_1}\tilde Du}\cdots\abs{D^{j_b}\tilde Du}\\
 &= I + II.
\end{aligned}
\end{equation}

For convenience, we denote $s = \tilde Du$. Then we can apply
Theorem~\ref{thm:int} on $s$ which is a section of the bundle
$u(t)^*TN$ on $M_1$ to get
\begin{equation}\label{eq:int}
\Norm{D^j s}_{L^p} \leq C\Norm{s}^a_{H^{k,q}}\Norm{s}^{1-a}_{L^r},
\end{equation}
where $1\leq p, q, r\leq \infty$ and $j/k \leq a \leq 1$ satisfy
\begin{equation}\label{eq:11}
  \frac 1 p = \frac j m + \frac 1 r + a(\frac 1 q - \frac 1 r - \frac k m).
\end{equation}

Let's first estimate the first term $I$ in (\ref{eq:13}). By
H\"older inequality,
\begin{equation}\label{eq:14}
I \le C\norm{D^l u}_{L^2}\norm{D^lD_t u}_{L^2} \le C\norm{D^{l-1}s}_{L^2}\norm{D^ls}_{L^2}.
\end{equation}
Then using the interpolation inequality (\ref{eq:int}), we have
\[ \norm{D^{l-1}s}_{L^2} \le C\Norm{s}^a_{H^{l,2}}\Norm{s}^{1-a}_{L^2}, \]
where $ a = (l-1)/l$ by (\ref{eq:11}). So we get
\begin{equation}\label{eq:15}
I \le C\Norm{s}^{(l-1)/l}_{H^{l,2}}\Norm{s}^{1/l}_{L^2}\norm{D^ls}_{L^2}.
\end{equation}

Next we treat the second term in~(\ref{eq:13}), i.e.
\[II = \int_{M_1}\abs{D^ls}\abs{D^{j_1}s}\cdots\abs{D^{j_b}s},\]
where the indices satisfy (\ref{equ:index}). Here we directly apply Ding-Wang's lemma in \cite{DW}.
Let $m_0 = [\frac m 2]+ 1$, where $[\frac m 2]$ is the integer part of $\frac m 2$.

\begin{lem}\label{lem:4}(\cite{DW})
If $1\le l \le m_0$, there exists a constant $C = C(M_1,l)$ such
that
\[ II \le C\norm{s}^A_{H^{m_0,2}}\norm{s}^B_{L^2}\norm{D^l s}_{L^2}, \]
where $A = [l +3 + (m/2 - 1)b - m/2]/m_0$ and $B = b - A$.
\end{lem}

\begin{lem}\label{lem:5}(\cite{DW})
If $l > m_0$, there exists a constant $C = C(M_1,l)$ such that\\
(i)if $j_1 = l$,
\[ II \le C\norm{s}^{m/m_0}_{H^{m_0,2}}\norm{s}^{2-m/m_0}_{L^2}\norm{D^l s}^2_{L^2},\]
(ii) if $j_1 \le l$,
\[ II \le C(1+\norm{s}^2_{H^{l,2}})(1+\norm{s}^A_{H^{l-1,2}}),\]
where $A = A(m,l)$.
\end{lem}

Now we can prove our main lemma. Note that previous computations do not depend on the variational structure. But to get the bound on
energy, we need to assume that $JV = \nabla \Ld$ in the following context.

\begin{lem}\label{lem:6}
  Suppose $(u_0,u_1)\in C^\infty(M_1,TN)$, then there exists
  \[T = T(\norm{\nabla u_0}_{H^{m_0,2}}, \norm{u_1}_{H^{m_0,2}})>0\]
  independent of  $\ep \in (0,1]$, such that if $u_\ep \in C^\infty(M_1\times[0,T_\ep],N)$ is a solution to~(\ref{equ:3}),~(\ref{equ:ini'}), then
  $T_\ep \ge T$, and
\begin{equation}\label{eq:6}
 \norm{\tilde Du}_{H^{k,2}} \le C(\norm{\nabla u_0}_{H^{k,2}}, \norm{u_1}_{H^{k,2}}), \forall t\in[0,T],
 \end{equation}
 for all $k\ge m_0$.
\end{lem}
\begin{proof}
  We still denote $s = \ti Du$, then the energy functional in Lemma~\ref{lem:6} is $E(t) = \frac 1 2\norm{s}_{L^2}$.
  Since $\Ld$ is a smooth function on a compact manifold $N$, it's bounded. From the energy inequality ~(\ref{eq:12}), we have
  \begin{equation}
    \norm{s}_{L^2} = 2E(t) \le 2E(0) - 2\int_{M_1}\Ld(u(t)) + 2\int_{M_1}\Ld(u(0)) \le C.
  \end{equation}

Now we turn to (\ref{eq:13}). We first consider the case $1\le l \le
m_0$. According to (\ref{eq:15}) and Lemma \ref{lem:4}, we have
\begin{equation*}
\begin{aligned}
\frac 1 2 \dif{t} \Norm{D^ls}^2_{L^2(M_1)} &\le I + II \\
&\le C\Norm{s}^{(l-1)/l}_{H^{l,2}}\Norm{s}^{1/l}_{L^2}\norm{D^ls}_{L^2}
+ C\sum\norm{s}^A_{H^{m_0,2}}\norm{s}^B_{L^2}\norm{D^l s}_{L^2} \\
&\le C\Norm{s}^{(l-1)/l}_{H^{m_0,2}}\norm{D^ls}_{L^2} + C\sum\norm{s}^A_{H^{m_0,2}}\norm{D^l s}_{L^2}.
\end{aligned}
\end{equation*}
Summing this inequality from $l = 1$ to $l = m_0$, we get
\begin{equation*}
\begin{aligned}
\frac 1 2 \dif{t} \Norm{s}^2_{H^{m_0,2}} &\le
C(\sum_{l}\Norm{s}^{(l-1)/l}_{H^{m_0,2}} + \sum_{b,l}\norm{s}^{A(b,l)}_{H^{m_0,2}})\norm{s}_{H^{m_0,2}}.
\end{aligned}
\end{equation*}
i.e.
\begin{equation}
  \dif{t} \Norm{s}_{H^{m_0,2}} \le
C(\sum_{l}\Norm{s}^{(l-1)/l}_{H^{m_0,2}} + \sum_{b,l}\norm{s}^{A(b,l)}_{H^{m_0,2}}).
\end{equation}
where
\[ A(b,l) = [l +3 + (m/2 - 1)b - m/2]/m_0. \]
If we let
\[ f(t) = \norm{s}_{H^{m_0,2}} + 1, \]
we have
\begin{equation}
  \left\{
  \begin{aligned}
  &\dif{t} f(t) \le Cf(t)^{A_0},\\
  &f(0) = \norm{\nabla u_0}_{H^{m_0,2}} + \norm{u_1}_{H^{m_0,2}} + 1.
  \end{aligned}
  \right.
\end{equation}
where
\[ A_0 = \max_{b,l}\{(l-1)/l, A(b,l)\}, \]
and the constant $C$ only depends on $\norm{\nabla u_0}_{H^{m_0,2}},
\norm{u_1}_{H^{m_0,2}}$ and the manifolds $M_1,N$.

It follows from ordinary differential equation theory that there exists
\[T = T(\norm{\nabla u_0}_{H^{m_0,2}}, \norm{u_1}_{H^{m_0,2}})>0\]
and a constant $K$ such that $f(t) \le K$, i.e.
\begin{equation}\label{eq:16}
  \norm{\ti Du(t)}_{H^{m_0,2}} \le K, \forall t \in [0,T].
\end{equation}

Next we treat the case $k > m_0$. (\ref{eq:13}),~(\ref{eq:14})
together with Lemma~\ref{lem:5} leads to
\[ \dif{t} \norm{D^l s}^2_{L^2} \le C\norm{s}^2_{H^{k,2}} + C\sum(1+\norm{s}^2_{H^{k,2}})(1+\norm{s}^A_{H^{k-1,2}}). \]
Summing up from $l = 1, \cdots, k$, we get
\begin{equation}\label{eq:17}
  \dif{t} \norm{s}^2_{H^{k,2}} \le C\sum(1+\norm{s}^2_{H^{k,2}})(1+\norm{s}^A_{H^{k-1,2}}).
\end{equation}
Then we perform a induction for $k>m_0$. Specifically, we first consider $k = m_0 +1$.
 From~(\ref{eq:16}), (\ref{eq:17}), we get
\[ \dif{t} \norm{s}^2_{H^{m_0+1,2}} \le CK\sum(1+\norm{s}^2_{H^{m_0+1,2}}), \forall t \in [0,T].\]
By Gronwall's inequality, we get
\[ \norm{\ti Du(t)}_{H^{m_0+1,2}} \le C', \forall t \in [0,T]. \]
Then by induction, for any $k = m_0+i, i \ge 1$ it follows from ~(\ref{eq:16}),~(\ref{eq:17}) that
\[ \norm{\ti Du(t)}_{H^{k,2}} \le C_k, \forall t \in [0,T]. \]
where $C_k$ only depends on $\norm{\nabla u_0}_{H^{k,2}},
\norm{u_1}_{H^{k,2}}$.

Thus we proved the lemma.
\end{proof}

Now we can prove the local existence of the solution to Cauchy problem
(\ref{equ:2'}),(\ref{equ:ini}) with smooth initial data.

\begin{thm}\label{thm:smooth}
Suppose $(u_0, u_1)\in C^\infty(M_1,TN)$, then there exists
\[T = T(\norm{\nabla u_0}_{H^{m_0,2}}, \norm{u_1}_{H^{m_0,2}})>0\]
such that  the Cauchy problem
 (\ref{equ:2'}),(\ref{equ:ini}) has a local solution $u\in C^\infty(M_1\times [0,T],N)$.
\end{thm}
\begin{proof}
For any $\ep>0$, there is a smooth solution $u_\ep \in C^\infty(M_1\times [0,T_\ep], N)$ to ~(\ref{equ:3}).
Moreover, $u_\ep$ satisfies the estimate (\ref{eq:6}) in Lemma~\ref{lem:6} and there is a constant $T>0$ such that $T_\epsilon \ge T,\forall \ep >0$.
It follows form Lemma ~\ref{lem:equ} that
\begin{equation}\label{eq:7}
  \max_{t\in [0,T]} \norm{\ti D u_\ep}_{W^{k,2}} \le C(\norm{\nabla u_0}_{H^{k,2}}, \norm{u_1}_{H^{k,2}}), \forall k\ge m_0,
\end{equation}
where the constant $C$ is independent of $\ep$. Thus,
by letting $\epsilon \to 0$ and applying Sobolev embedding theorems,
we can find a limit map $u\in C^\infty(M_1\times [0,T], N)$, such
that $u_\ep \to u$ in $C^k(M_1\times [0,T], N)$ for any $k$.
It's easy to verify that $u$ is a smooth solution to equation~(\ref{equ:2}).
\end{proof}

From (\ref{eq:7}), one can easily see that the limit map $u$ also satisfies the same estimate, i.e.
\begin{equation*}
  \max_{t\in [0,T]} \norm{\ti D u}_{W^{k,2}} \le C(\norm{\nabla u_0}_{H^{k,2}}, \norm{u_1}_{H^{k,2}}), \forall k\ge m_0,
\end{equation*}
In fact, we can say more about $u$. Namely, the above
inequality also holds for $k=m_0-1$.

\begin{lem}\label{lem:7}
  Suppose $u$ is a solution to Cauchy problem~(\ref{equ:2'}),(\ref{equ:ini})
  given by Theorem~\ref{thm:smooth}, then
\begin{equation}\label{eq:8}
 \max_{t\in [0,T]} \norm{\tilde Du}_{H^{k,2}} \le C(\norm{\nabla u_0}_{H^{k,2}},
 \norm{u_1}_{H^{k,2}}), \forall k\ge m_0-1.
 \end{equation}
\end{lem}
\begin{proof}
The proof goes almost the same with the proof of Lemma~\ref{lem:6}, except for a more
refined estimate on the curvature term. The observation is that without the approximating
term $\ep DD u_t$, there are only three terms left in the curvature term (\ref{equ:curv}).
Indeed, this term becomes
\begin{equation*}
\begin{split}
\abs{\ti Q} &\le |Q_1| + |Q_2| + |Q_3|\le C\abs{Q(\tilde Du, \tilde Du)}\\
&\le C\sum\abs{D^{j_1}\tilde Du}\cdots\abs{D^{j_b}\tilde Du},
\end{split}
\end{equation*}
where the summation is now over all indexes $(j_1, \cdots, j_b)$ satisfying
\begin{equation}\label{equ:index1}
j_1\ge j_2 \ge \cdots \ge j_b,\quad l\ge j_i \ge 0,\quad j_1 + \cdots + j_b + b \le l+2,\quad b \ge 3.
\end{equation}
The key is that the sum of the index in (\ref{equ:index1}) is $l+2$, which is one order lower than $l+3$ in (\ref{equ:index}). With this change, one can verify that all the estimates in
the rest part of proof of Lemma~\ref{lem:6} holds for $m_0-1$ instead of $m_0$.
\end{proof}

Now we are ready to prove the main theorem.

\begin{thm}
Suppose $u_0\in W^{k,2}(M_1,N), u_1\in W^{k-1,2}(M_1,T_{u_0}N)$,
where $k\ge m_0$. Then the Cauchy problem
 (\ref{equ:2'}),(\ref{equ:ini}) has a local solution $u\in L^\infty([0,T],W^{k,2}(M_1,N))$
with $u_t\in L^\infty([0,T],W^{k-1,2}(M_1,TN))$.
\end{thm}

\begin{proof}
Since $u_0\in W^{k,2}(M_1,N)$ with $k\ge m_0$ larger than the
borderline $m/2$ for Sobolev imbedding into $C^0(M_1,N)$, we can
approximate $u_0$ by smooth maps in $C^\infty(M_1,N)$(see \cite{DW}
for a proof). Namely, we may select a sequence $(u_0^i, u_1^i)\in
C^\infty(M_1,TN)$, such that
$$u_0^i \to u_0 \mbox{~in~} W^{k,2}(M_1, N),
u_1^i \to u_1 \mbox{~in~} W^{k-1,2}(M_1, \Real^{2K}).$$ Then for any
$i\ge 1$ and initial data $(u_0^i, u_1^i)$, there exits a local
solution $u^i$ which satisfies (\ref{eq:8}). Since as $i\to \infty$
\begin{align*}
  &\norm{u_0^i}_{W^{k,2}} \to \norm{u_0}_{W^{k,2}} \\
  &\norm{u_1^i}_{W^{k-1,2}} \to \norm{u_1}_{W^{k-1,2}},
\end{align*}
the estimate (\ref{eq:8}) is uniform with respect to $i$ and only
depends on $\norm{u_0}_{W^{k,2}}$ and $\norm{u_1}_{W^{k-1,2}}$.
Hence
\begin{align}
  \label{eq:9}&\max_{t\in [0,T]} \norm{u^i}_{W^{k,2}} \le C(\norm{u_0}_{W^{k,2}},\norm{u_1}_{W^{k-1,2}}), \\
  \label{eq:10}&\max_{t\in [0,T]} \norm{u^i_t}_{W^{k-1,2}} \le  C(\norm{u_0}_{W^{k,2}},\norm{u_1}_{W^{k-1,2}}).
\end{align}
Therefore we can find a subsequence which we still denote by $u^i$,
such that
\begin{equation*}
\begin{split}
 u^i\rightharpoonup u &\text{ ~~in~~ } L^\infty([0,T],W^{k,2}(M_1,N)), \\
 u^i_t\rightharpoonup u_t &\text{ ~~in~~ } L^\infty([0,T],W^{k-1,2}(M_1,TN))
\end{split}
\end{equation*}
where $\rightharpoonup$ denotes the weak $*$ convergence.

The limit $u$ is a strong solution to ~(\ref{equ:2}). To show this
we only have to verify that for any $v\in
C^\infty(M_1\times[0,T],\rk)$, there holds
\begin{equation}\label{e11}
  \int_0^T \int_{M_1} \lg \square u - A(u)(\ti \nabla u,
  \ti \nabla u), v \rg = - \int_0^T \int_{M_1} \lg J(u)V(u), v \rg.
\end{equation}
Indeed, since $u^i$ is a solution, we have
\begin{equation}\label{e15}
  \int_0^T \int_{M_1} \lg \square u^i - A(u^i)(\ti \nabla u^i, \ti \nabla u^i), v \rg = - \int_0^T \int_{M_1} \lg J(u^i)V(u^i), v \rg,
\end{equation}
And the estimates ~(\ref{eq:9}),~(\ref{eq:10}) holds true. So we have
\[\max_{t\in [0,T]}\norm{\ti \nabla u^i}_{W^{k-1,2}} = \max_{t\in [0,T]}\norm{u^i_t + \nabla u^i}_{W^{k-1,2}} \le C. \]
when $k\ge m_0 +1$, by Sobolev, we know that for all $t \in [0,T]$
\begin{equation}\label{e9}
  \ti \nabla u^i \to \ti \nabla u \text{ ~~in~~ } C^0(M_1,N).
\end{equation}
and
\begin{equation}\label{e10}
  \Delta u^i \to \Delta u \text{ ~~in~~ } L^\infty([0,T],L^2(M_1,N)).
\end{equation}
The above convergence implies
\begin{equation}\label{e12}
  \lim_{i\to \infty}\int_0^T \int_{M_1} \lg -\Delta u^i - A(u^i)(\ti \nabla u^i, \ti \nabla u^i), v \rg =
\int_0^T \int_{M_1} \lg -\Delta u - A(u)(\ti \nabla u, \ti \nabla u), v \rg,
\end{equation}
and
\begin{equation}\label{e13}
 \lim_{i \to \infty} \int_0^T \int_{M_1} \lg J(u^i)V(u^i), v \rg = \int_0^T \int_{M_1} \lg J(u)V(u), v \rg.
\end{equation}
On the other hand, we have
\begin{equation}\label{e14}
\lim_{i\to \infty} \int_0^T\int_{M_1} \lg u^i_{tt}, v \rg = -\int_0^T\int_{M_1} \lg u_{t}, v_t \rg + \int_{M_1}(\lg u_t(T), v(T)\rg - \lg u_t(0), v(0)\rg).
\end{equation}
Now we can deduce from ~(\ref{e15}),~(\ref{e12}),~(\ref{e13}) and ~(\ref{e14}) that
\begin{align*}
  -\int_0^T\int_{M_1} \lg u_{t}, v_t \rg + &\int_{M_1}(\lg u_t(T), v(T)\rg - \lg u_t(0), v(0)\rg) =\\
 &\int_0^T \int_{M_1} \lg \Delta u + A(u)(\ti \nabla u, \ti \nabla u), v \rg - \int_0^T \int_{M_1} \lg J(u)V(u), v \rg.
\end{align*}
This means $u_{tt} \in L^2([0,T]\times M_1, N)$, so we have proved~(\ref{e11}), hence the theorem.

\end{proof}

Finally, we prove the uniqueness of the local solution. If $u,v$ are
two solutions to Cauchy problem~(\ref{equ:2'}),~(\ref{equ:ini}), we
need to show $u = v$. Generally, one may consider the difference
$u-v$ between $u$ and $v$. But in order to do the substraction, one
needs to consider the embedding $N \hookrightarrow \rk$. The
following computation also relies on such an embedding.

\begin{thm}\label{thm:unique}
Suppose $u_0\in W^{k,2}(M_1,N), u_1\in W^{k-1,2}(M_1,T_{u_0}N)$,
where $k\ge m_0$ for $m\ge 2$ and $k=2$ for $m=1$. Then the local
solution to (\ref{equ:2'}), (\ref{equ:ini}) is unique in class
$W^{k,2}$.
\end{thm}
\begin{proof}
  Assume $u,v$ are two local solutions to (\ref{equ:2'}),
  (\ref{equ:ini}) satisfying
\begin{equation*}
  u,v\in L^\infty([0,T],W^{k,2}(M_1,N));\quad
  u_t,v_t\in L^\infty([0,T],W^{k-1,2}(M_1,TN)).
\end{equation*}
  Since we embed $N$ into a Euclidean space
  $\rk$, we can compute
  \begin{align*}
    &\quad~~\frac 1 2 \dif{t}\norm{\ti D(u - v)}^2_{L^2} \\
    &= \int_{M_1} \lg D_t(u-v), D_t^2(u-v)\rg - \lg D_t(u-v), \Delta(u-v) \rg \\
    &= \int_{M_1} \lg D_t(u-v), (A(u)(\ti Du, \ti Du)-A(v)(\ti Dv, \ti Dv))  -(J(u)V(u)-J(v)V(v)) \rg \\
    &= \int_{M_1} \lg u_t, (A(u)(\ti Dv, \ti Dv)-A(v)(\ti Dv, \ti Dv))\rg - \lg v_t,A(u)(\ti Du, \ti Du) - A(v)(\ti Du, \ti Du)\rg\\
            &~~~~~~~~~~~~~~~~~~~~~~~~~~~~~~~~~~~~~~~~~~~~~~~~~~~~~~+ \lg u_t-v_t,  -(J(u)V(u)-J(v)V(v)) \rg \\
    &\le \int_{M_1} |A(u)-A(v)|(\lg u_t,|\ti Dv|^2\rg - \lg v_t,|\ti Du|^2\rg) + C\int_{M_1}|u_t-v_t||u-v|\\
    &\le C\int_{M_1}|u-v||\ti Du-\ti Dv|(|\ti Du|^2+|\ti Dv|^2) + C\int_{M_1}|u_t-v_t||u-v|\\
    &\le C\norm{\ti Du-\ti Dv}_{L^2}\cdot(\norm{|u-v|(|\ti Du|^2+|\ti
    Dv|^2)}_{L^2}+\norm{u-v}_{L^2}).
  \end{align*}
  Hence we get
  \begin{equation}
    \dif{t}\norm{\ti D(u - v)}_{L^2} \le C(\norm{|u-v|(|\ti Du|^2+|\ti
    Dv|^2)}_{L^2}+\norm{u-v}_{L^2}).
  \end{equation}

  If $m\le 3$, we have $k \ge 2$. By Sobolev embedding $W^{2,2} \hookrightarrow W^{1,6}$, we get
    \begin{equation}\label{eq:18}
    \begin{aligned}
    \dif{t}\norm{\ti D(u - v)}_{L^2} &\le C\norm{u-v}_{L^6}(\norm{\ti Du}^2_{L^6}+\norm{\ti Dv}^2_{L^6})+C\norm{u-v}_{L^2}\\
    &\le C\norm{\ti Du-\ti Dv}_{L^2}(\norm{\ti Du}_{W^{1,2}} + \norm{\ti
    Dv}_{W^{1,2}}).
  \end{aligned}
  \end{equation}
  If $m>3$, we have Sobolev embedding $W^{[\frac{m}{2}]+1,2} \hookrightarrow W^{1,2m}$. Thus
  \begin{equation}\label{eq:19}
    \begin{aligned}
    \dif{t}\norm{\ti D(u - v)}_{L^2} &\le C(\norm{u-v}_{L^{\frac{2m}{m-2}}}(\norm{\ti Du}^2_{L^{2m}}+\norm{\ti Dv}^2_{L^{2m}})+C\norm{u-v}_{L^2}\\
    &\le C\norm{\ti Du-\ti Dv}_{L^2}(\norm{\ti Du}_{W^{[\frac{m}{2}],2}} + \norm{\ti
    Dv}_{W^{[\frac{m}{2}],2}}+1).
    \end{aligned}
  \end{equation}

    From ~(\ref{eq:18}),~(\ref{eq:19}) and Lemma~\ref{lem:6}, it follows that, if $(u_0, u_1)\in W^{k,2}(M_1,N)\times W^{k-1,2}(M_1,TN)$, there holds
    \[ \dif{t}\norm{\ti D(u - v)}_{L^2} \le C\norm{\ti Du-\ti Dv}_{L^2}. \]
    By Gronwall's inequality, we finally get
    \[ \norm{\ti D(u(t) - v(t))}_{L^2} \le C\norm{\ti D(u(0) - v(0))}_{L^2} = 0. \]
    Thus we complete the proof.
\end{proof}

\begin{rem}
We can also compute the difference between $u$ and $v$ intrinsically
by using parallel translation. Mcgahahan~\cite{Mc} used this method
to prove the continuous dependence of solutions to \sch flow on
initial data. Same method can by applied to prove continuous
dependence of initial data to Cauchy
problem~(\ref{equ:2'}),~(\ref{equ:ini}).
\end{rem}

\begin{rem}
  We can also consider \sch flow with potential, i.e.
  \begin{equation}\label{eq:schp}
    \ut = J(u)\tau(u) + J(u)\nabla F(u),
  \end{equation}
  where $F$ is a smooth function. Actually, we can prove the local existence of~(\ref{eq:schp}) by the same method.
\end{rem}

\section{Global existence in $1+1$ dimension}

In this section, we follow the method in~\cite{SS} to prove
Theorem\ref{thm:global}. Note that when $m=1$, $m_0 = 1$ and $k\ge
2$ in Theorem \ref{thm:sch}.

\begin{proof}[Proof of Theorem~\ref{thm:global}]
  According to Theorem~\ref{thm:sch}, we already have a unique local solution $u\in L^\infty([0,T),W^{2,2}(S^1,N))$. Moreover, $u$ satisfies the
  estimate~(\ref{eq:6}). Now we need to derive a global estimate. Since $u$ satisfies equation~(\ref{equ:2'}), i.e.
  \begin{equation}\label{e2}
    \square u = A(u)(\ti Du,\ti Du) - J(u)V(u).
  \end{equation}
  Applying a first order spatial derivative $\nabla$ to this equation, we get
  \begin{align*}
    \square (\nabla u) &= \nabla(A(u)(\ti Du,\ti Du)) - \nabla(J(u)V(u))\\
    &= \nabla A(u)(\ti Du, \ti Du, \nabla u) + 2A(u)(\nabla\ti Du, \ti Du)
    - J(u)\nabla V(u)\cdot \nabla u.
  \end{align*}
  But for the second fundamental form $A$, we have
  \[ \lg u_t, A(u)(\cdot, \cdot) \rg =0.\]
  Thus
  \begin{equation}\label{e3}
    \lg \nabla u_t, A(u)(\nabla\ti Du, \ti Du) \rg = \lg u_t, \nabla A(u)(\nabla\ti Du, \ti Du, \nabla u) \rg.
  \end{equation}
  The above equality implies
  \begin{eqnarray}\label{eq:20}
    \nonumber \frac 1 2 \dif{t} \norm{\nabla\ti Du}^2_{L^2} &=& \int_{M_1} \lg \square (\nabla u), \nabla u_t \rg\\
    \nonumber &=& \int_{M_1} \lg \nabla A(u)(\ti Du, \ti Du, \nabla u) + 2A(u)(\nabla \ti Du, \ti Du),\nabla u_t \rg\\
        \nonumber&~~&~~~ - \lg J(u)\nabla V(u)\cdot \nabla u, \nabla u_t \rg\\
    \nonumber &= &\int_{M_1} \lg \nabla A(u)(\ti Du, \ti Du, \nabla u),\nabla u_t \rg - \lg J(u)\nabla V(u)\cdot \nabla u, \nabla u_t \rg \\
        \nonumber&~~&~~~+ 2\int_{M_1} \lg u_t, \nabla A(u)(\nabla \ti Du, \ti Du, \nabla u)\\
    &\le& C\int_{M_1} |\ti Du|^3|\nabla \ti Du| + |\ti Du||\nabla \ti
    Du|.
  \end{eqnarray}
  From H\"older's inequality,
    \begin{equation}\label{eq:21}
        \int_{M_1} |\ti Du|^3|\nabla \ti Du| \le C\norm{\ti Du}^3_{L^6}\norm{\nabla \ti
        Du}_{L^2}.
    \end{equation}
  When $m = 1$, it follows from the classic Gagliardo-Nirenberg interpolation inequality and Kato's inequality that
  \begin{equation}\label{eq:22}
    \norm{\ti Du}_{L^6} \le \norm{\nabla \ti Du}^a_{L^2}\norm{\ti
    Du}^{(1-a)}_{L^2},
  \end{equation}
  where
    \begin{equation*}
    \frac 1 6 = a(\frac 12 - 1) + (1-a)\frac 12.
  \end{equation*}
  i.e. $a = \frac 1 3$. Hence we arrive at a Gronwall-type inequality from~(\ref{eq:20}),~(\ref{eq:21}) and ~(\ref{eq:22})
  \begin{equation*}
    \dif{t} \norm{\nabla \ti Du}^2_{L^2} \le \norm{\nabla \ti Du}^2_{L^2}\norm{\ti Du}^2_{L^2}
  \end{equation*}
  Combining this together with the energy inequality $\norm{\ti Du}^2_{L^2} \le C$, we obtain
   \begin{equation}\label{eq:23}
     \norm{\nabla \ti Du}^2_{L^2} \le C(t), \forall t\in \Real.
   \end{equation}

  Now we can derive the global existence from Theorem~\ref{thm:sch} and ~(\ref{eq:23}).
  Indeed, if this is not the case, assume the maximal existence time interval of $u$ is $[0,T)$. It follows from Lemma~\ref{lem:6} that $T$ only depends on the initial data, i.e.
  \[T = T(\norm{\ti D u(0)}_{H^{1,2}}).\]
  We may choose a small positive number $\ep>0$, and consider the Cauchy problem~(\ref{equ:3}) with initial data $u(T-\ep)$.
  Then Theorem~\ref{thm:sch} guarantees the existence of another local solution
  $u' \in L^\infty([0,T'),W^{2,2}(S^1,N))$, where
  \[ T' = T(\norm{\ti D u(T-\ep)}_{H^{1,2}}). \]
  Moreover, by the uniqueness Theorem~\ref{thm:unique}, $u$ and $u'$
  coincides on the overlapped time interval. Now, if we patch $u, u'$ together,
  we get a solution to (\ref{equ:2'}),(\ref{equ:ini}) on the time interval $[0,T-\ep + T')$.
  The estimate (\ref{eq:23}) tells us that $\norm{\ti D u(t)}_{H^{1,2}}$ is uniformly bounded for all $t\in [0,T)$.
  Consequently, if $\ep$ is small enough, we have $T-\ep + T' >T$. This contradicts to the maximality of $T$.
  Hence, we must have $T = \infty$.
\end{proof}

\appendix

\section{Local existence of the approximation}

In this appendix, we use a fixed point argument to prove the local
existence of the Cauchy problem of equation~(\ref{equ:4}):
\begin{equation}\label{a:1}
\left\{
\begin{aligned}
&u_{tt} - \ep \Delta u_t = F(u,u_t)\\
&u(0) = u_0, u_t(0) = u_1
\end{aligned}\right.
\end{equation}
where
\[F(u,u_t) = \Delta u - J(u)V(u) +A(u)(\nabla u+u_t,\nabla
u+u_t) - \ep T(u)(\Delta u_t)\] and
$$ u_0\in C^\infty(M_1, N), u_1\in C^\infty(M_1, TN)$$
satisfy the following condition:
$$u_1(x) \in T_{u_0(x)}N,  \forall x\in M_1.$$

Consider the Banach spaces
$$ X = \{ v=(v_1,v_2)\in C^3(M_1,N)\times C^2(M_1,TN); v_2(x) \in T_{v_1(x)}N,  \forall x\in M_1\} $$
with the norm
\[ \norm{v}_X = \norm{v_1}_{C^3(M_1)} + \norm{v_2}_{C^2(M_1)} \]
and
$$ Y = C^1(M_1, N) $$
with the norm
\[ \norm{f}_Y = \norm{f}_{C^1(M_1)}. \]
We recall the expression of $T(u)(\Delta u_t)$ given by~(\ref{a}),
i.e.
\begin{equation}
T(u)(\Delta u_t) = A(u)(\nabla u_t, \nabla u) + \text{div}
(A(u)(u_t, \nabla u)).
\end{equation}
From this equality, one can see that if $(u,u_t) \in X = C^3\times
C^2$, then $F(u,u_t)\in C^1$. Therefore, $F$ is a mapping from $X$
into $Y$. In fact, we have

\begin{lem}\label{a1}
$F$ is a locally Lipschitz map from $X$ to $Y$.
\end{lem}
\begin{proof}
For any $v=(v_1,v_2), w=(w_1,w_2)\in X$, we have
\begin{align*}
\norm{F(v) - F(w)}_Y &\le \norm{\Delta v_1- \Delta w_1 + J(v_1)V(v_1) - J(w_1)V(w_1)}_Y\\
&~~~~ + \norm{A(v_1)(\nabla v_1 + v_2,\nabla v_1 + v_2) -
A(w_1)(\nabla w_1 + w_2,\nabla w_1 + w_2)}_Y\\ &~~~~  + \ep \norm{T(v_1)(\Delta v_2) - T(w_1)(\Delta w_2)}_Y\\
&\le I+II+III.
\end{align*}
Obviously, we have
\[ I \le \norm{v-w}_{X}. \]
For the second fundamental form,
\begin{align*}
II &\le \norm{|A(v_1) - A(w_1)||\nabla v_1+v_2|^2}_Y \\&~~~~+
\norm{A(w_1)(|\nabla v_1-\nabla w_1|+ |v_2-w_2|)(|\nabla
v+v_2|+|\nabla w+w_2|)}_Y\\
&\le C(\norm{v}_X^2 + \norm{w}_X)\norm{v-w}_X.
\end{align*}
As for the third term, by a similar computation, we have
\begin{align*}
III &\le \ep \norm{A(v_1)(\nabla v_2, \nabla v_1)- A(w_1)(\nabla
w_2, \nabla w_1)}_Y\\
&~~~~~~~~~~~~~~~~~~~~~ + \ep\norm{\text{div}
(A(v_1)(v_2, \nabla v_1) - A(w_1)(w_2, \nabla w_1))}_Y\\
&\le C(\norm{v}_X^2 + \norm{w}_X^2)\norm{v-w}_X.
\end{align*}
Thus we obtain
\[ \norm{F(v) - F(w)}_Y \le C(1+\norm{v}_X^2 + \norm{w}_X^2)\norm{v-w}_X, \]
which means $F$ is locally Lipschitz.
\end{proof}

It's well-known that there exists a heat kernel on compact manifold
$M_1$, which we denote by $H(x,y,t)$. We first fix $u\in X$. Using
the heat kernel, one can solve the linear parabolic equation
\begin{equation}
\left\{
\begin{aligned}
&v_{t} - \ep \Delta v = F(u)\\
&v(0) = u_1
\end{aligned}\right.
\end{equation}
by
\[ v(x,t)=\Psi(u) = \int_{M_1}H(x,y,\ep t)u_1(y)dy + \int_0^t\int_{M_1}H(x,y,\ep(t-s))F(u(y))dydt. \]
Then one can go on to solve an ordinary equation
\begin{equation}
\left\{
\begin{aligned}
&w_t = \Psi(u)\\
&w(0) = u_0.
\end{aligned}\right.
\end{equation}
The solution is given by $$w(t) = \Phi(u) = \int_0^t \Psi(u)(s)ds +
u_0.$$

Now we are ready to derive a fixed point argument. Fix $\delta>0$,
and set \begin{align*}
Z &= \{ u\in C([0,T],C^3(M_1))\cap C^1([0,T],C^2(M_1));\\
&~~~~~~~~~~~~~~~~~~~~~~~(u,u_t)|_{t=0} =
(u_0,u_1),\norm{(u(t),u_t(t))-(u_0,u_1)}_X \le \delta \}
\end{align*}
with the norm
$$\norm{u}_Z = \sup_{t\in[0,T]}\norm{(u(t),u_t(t))}_X.$$

\begin{lem}
$\Phi:Z\to Z$ is a contraction if $T$ is sufficiently small.
\end{lem}
\begin{proof}
For any $u,v\in Z$, we use the estimates of the heat kernel and
Lemma~\ref{a1} to get
\begin{align*}
\norm{\Phi(u)-\Phi(v)}_Z &\le \sup_{t\in[0,T]}\norm{(\Phi(u)-\Phi(v),\Psi(u)-\Psi(v))}_X\\
&\le \sup_{t\in[0,T]}(\norm{\int_0^t \Psi(u(s))-\Psi(v(s))ds}_{C^3}+ \norm{\Psi(u(t))-\Psi(v(t)}_{C^2})\\
&\le \sup_{t\in[0,T]}\int_0^t Ct^{-\alpha}\norm{F(u,u_t)-F(v,v_t)}_Y ds\\
&\le CT^{1-\alpha}\delta\sup_{t\in[0,T]} \norm{u(t)-v(t)}_X,
\end{align*}
where $\alpha\in(0,1)$ is a constant. Clearly if $T$ is small,
$\Phi$ is a contraction of $Z$.
\end{proof}

Then by the Banach fixed point theorem, $\Phi$ has a unique fixed
point $u\in Z$, which is a local solution to equation~(\ref{a:1}).
The regularity can be easily deduced from the property of the heat
kernel.


\end{document}